\documentclass{amsart}

\usepackage{amscd}
\usepackage{amsmath}
\usepackage{amssymb}
\usepackage{amsthm}
\usepackage{graphicx}
\usepackage{setspace}  
\usepackage[all,cmtip]{xy}

\theoremstyle{plain}
\newtheorem{thm}{Theorem}[section]
\newtheorem{lem}[thm]{Lemma}
\newtheorem{prop}[thm]{Proposition}

\newtheorem{rem}[thm]{Remark}

\newtheorem*{conj}{Conjecture}

\newcommand{\Sth}{S^3}

\newcommand{\Hth}{\mathbb{H}^3}
\newcommand{\PSLTC}{PSL(2, \mathbb{C})}

\newcommand{\PSLTOT}{PSL(2, \mathbb{O}_3)}
\newcommand{\PGLTOT}{PGL(2, \mathbb{O}_3)}

\newcommand{\Z}{\mathbb{Z}}

\newcommand{\Q}{\mathbb{Q}}

\newcommand{\bnm}{\beta_{n,m}}
\newcommand{\onm}{\Omega_{n,m}}
\begin{document}

\title{Commensurability classes containing three knot complements}

\author{Neil Hoffman}
\address{Department of Mathematics\\
Universtiy of Texas\\ 
Austin, TX 78712, U.S.A.}
\email{nhoffman@math.utexas.edu}

\date{\today}

\begin{abstract}
This paper exhibits an infinite family of hyperbolic knot 
complements that have three knot complements in their 
respective commensurability 
classes.
\end{abstract}

\maketitle

\section{Introduction}\label{intro}

The study of the commensurability classes of hyperbolic knot complements that
contain other knot complements has attracted some recent interest 
(see \cite{BBW},\cite{CD}, \cite{GHH}
\cite{HS},\cite{MM}, \cite{NR1},\cite{Reid}, \cite{RW}). A particularly 
interesting set of examples results from cyclic 
surgeries on hyperbolic knot complements, since the 
cyclic surgeries
give rise to cyclic covers 
by other knot complements (see \cite{GW}). Moreover, The Cyclic Surgery Theorem 
\cite{CGLS} shows that there are at most two non-trivial cyclic 
surgeries on a hyperbolic knot complement and so a hyperbolic 
knot complement has at most two non-trivial, 
finite sheeted covers which are  
other knot complements.  Similarly, if a hyperbolic
knot complement, $S^3-k_1$ is covered by another knot complement, $S^3-k_2$,
then $S^3-k_1$ admits a cyclic surgery. There are known examples of hyperbolic 
knot complements with exactly three knot complements 
in their commensurability classes. 
For example, the $(-2,3,7)$ pretzel knot of \cite{FS} famously 
admits two non-trivial cyclic surgeries and is therefore covered by 
two other hyperbolic knot complements.

An infinite family of pairs of commensurable hyperbolic knot complements 
was constructed by W. Neuman. 

For a discussion of this construction, see  \cite{GHH}.

Finally, two hyperbolic knot complements can be commensurable if they both
have hidden symmetries. This property is equivalent to 
both knot complements non-normally covering 
the same orbifold (see $\S$ \ref{prelim_cusps}). 
The dodecahedral knots of \cite{AR} 
admit the only known examples of non-arithmetic knot complements with hidden
symmetries (see \cite{NR1})
and the figure 8 knot complement is the only arithmetic knot complement 
(see \cite{Reid}).  

This discussion motivates
the following conjecture of Reid and Walsh 
(see \cite[Conj 5.2]{RW}).

\begin{conj}
Let $\Sth - K$ be a hyperbolic knot complement. Then, 
there are at most two other knot complements
in its commensurability class.
\end{conj}

It has been announced by Boileau, Boyer, and Walsh (\cite[Thm 1.3]{BBW}) 
that the conjecture holds for knot complements without hidden symmetries. 
In their paper, they show that if a hyperbolic knot complement does not admit hidden
symmetries, then any commensurable hyperbolic knot complement will
cover a common orbifold. Furthermore, this orbifold admits
a finite cyclic surgery for each knot complement that covers it. This paper
presents a family of such orbifolds that are covered by exactly three
hyperbolic knot complements.
Specifically, the main theorem 
of this paper is 
the following (see $\S$ \ref{prelim} for definitions):

\begin{thm}\label{mainthm1}
Let $n\geq 1$ and $(n,7)=1$. For all but at most
finitely many pairs of integers $(n,m)$, the result of $(n,m)$ Dehn surgery 
on the unknotted cusp of the Berge manifold is a 
hyperbolic orbifold with 
exactly three knot complements its commensurability classes.
\end{thm}

The infinite family of orbifolds described by Theorem \ref{mainthm1} 
which we refer to as $\bnm$ 
(see $\S$\ref{prelim}) also has 
the property that for $n \ne 1$, each knot complement covering 
$\bnm$ admits an $n$-fold symmetry which does not fix any point on the
cusp. In particular, even when $n=2$, this symmetry is not a 
strong involution.  
By \cite{WZ}, such a knot complement cannot 
admit a lens space surgery and so, by the above discussion, is not 
covered by any other knot complement.

The paper is organized as follows. In addition to some background material
and definitions, 
$\S$ \ref{prelim} we prove a lemma about possible orbifold quotients
of the Berge manifold.
In $\S$ \ref{threeCyclicSurgeries}, we show that the orbifolds $\bnm$ are 
shown to admit three cyclic surgeries, and   
the proof of the main theorem is contained 
in $\S$\ref{mainproof}. In $\S$ \ref{remarks},
we provide a partial classification of commensurability classes containing
three knot complements.

\section{Preliminaries}\label{prelim}
\subsection{}\label{prelim_com}
Two hyperbolic 3-orbifolds, $\Hth/\Gamma_1$ and $\Hth/\Gamma_2$, 
are said to be commensurable if they share a common finite 
sheeted cover. In terms of groups, $\exists g \in \PSLTC$ so
that    
$\Gamma_1$ and $g\Gamma_2 g^{-1}$ have a common subgroup which is
finite index in both groups.

Let $Comm^+(\Gamma) = \{g\in \PSLTC | [\Gamma : \Gamma \cap g\Gamma g ^{-1}]
< \infty \mbox { and }
[g\Gamma g^{-1} : \Gamma \cap g\Gamma g ^{-1}] < \infty \}$ and $N^+(\Gamma)$ be the normalizer of $\Gamma$ in $\PSLTC$. 
We say that a group $\Gamma$ has \emph{hidden symmetries} if
$[Comm^+(\Gamma):N^+(\Gamma)]>1$. A hyperbolic orbifold, M, 
has \emph{hidden symmetries} if $\pi_1^{orb}(M)$ has hidden symmetries. 
For this discussion, we consider only orientable manifolds and 
orbifolds.

\subsection{}\label{prelim_cusps}
When a hyperbolic knot group has hidden symmetries
the associated knot complement non-normally covers some 
orbifold with a \emph{rigid cusp} i.e. the cusp is $C\times[0,\infty)$ 
where $C$ is $S^2(2,3,6)$, 
$S^2(3,3,3)$ or $S^2(2,4,4)$ (see \cite[Lemma 4]{Reid}).

By \cite[Prop 2.7]{NR1}, the cusp field of a hyperbolic orbifold is a subfield
of the invariant trace field. Thus, if a hyperbolic 
orbifold has a $S^2(3,3,3)$ or
$S^2(2,3,6)$ cusp, $\Q(\sqrt{-3})$ must be a subfield of the orbifold's 
invariant trace field and if the cusp is $S^2(2,4,4)$, $\Q(i)$ 
must be a subfield of the orbifold's invariant trace field (see \cite[Proof of
Thm 5.1(iv)]{NR1}).

\begin{prop}\label{prop_degree}
Let $p: O_1 \rightarrow O_2$ be a covering of orbifolds 
such that $O_1$ has a rigid
cusp $C_1$. Then, $O_2$ has a rigid cusp 
$C_2$ such that $p(C_1) = C_2$
and if $x \in C_2$ then $|p^{-1}(x)\cap C_1|=n^2$ for some integer n
unless $C_1$ is $S^2(3,3,3)$ and $C_2$ is
$S^2(2,3,6)$ then $|p^{-1}(x)\cap C_1|=2n^2$ for some integer n.
\end{prop}

\begin{proof}
First consider the case where $C_1$ is an $S^2(2,4,4)$.
In this case, $C_2$ must also be a $S^2(2,4,4)$ cusp.
The peripheral subgroup corresponding to $C_2$ 
is $P_2 \cong (\Z\times\Z)\rtimes_{\phi} \Z/4\Z$, and so $P_2$ 
has an element of order 4 acting on the cusp. Thus,
$\phi: \Z/4\Z \rightarrow Aut(\Z \times \Z)$ 
is a faithful representation. Let $P_1 \subset P_2$ be the peripheral 
subgroup corresponding to $C_1$. So 
$P_1 \cong (n\Z \times m\Z)\rtimes_{\phi} \Z/4\Z$. 
However, the order 4 automorphism switches the two 
generators for the $\Z\times \Z$ subgroup of $P_2$. 
Thus, $n=m$ and the degree of the covering is $n^2$.

A similiar proof carries through if $C_1$ and $C_2$ are both 
either $S^2(3,3,3)$ or $S^2(2,3,6)$ cusps.

In the case, where $C_1$ is a $S^2(3,3,3)$ and $C_2$ is 
a $S^2(2,3,6)$ cusp, the $\Z/3\Z$ subgroup of $P_1$ is index 2 in the 
$\Z/6\Z$ subgroup of $P_2$. Hence, the covering degree is $2n^2$.   

\end{proof}

\subsection{}\label{bergeDef} 
For $n\geq 1$ and $(n,7)=1$, let $\bnm$ be the orbifold obtained by $(n,m)$ 
Dehn surgery on the unknotted cusp of the Berge manifold 
(see Figure \ref{BergeManifold}) using a standard framing on the cusps 
of this link complement as in \cite{Rolfsen}. 

\begin{figure}[ht]
\begin{center}
\includegraphics[height= 2 in]{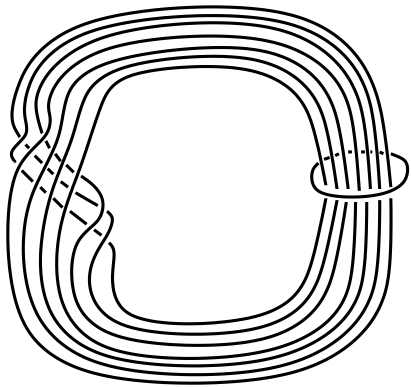}
\caption{\label{BergeManifold} The Berge manifold is the complement of this
link.}
\end{center}
\end{figure}

The Berge manifold admits several surgery slopes of 
interest. First if we perform Dehn surgery along the $(1,0)$ slope of 
the unknotted cusp of the 
Berge manifold, we will obtain the $(-2,3,7)$ pretzel knot (see \cite{FS}). 
Also, if we 
drill out a solid solid torus along the unknotted 
cusp of the manifold we would obtain
the one of the knots in the solid torus that admits three $D^2\times S^1$ 
fillings (see \cite[Cor 2.9]{Berge}). Furthermore, if we perform Dehn surgery
along the $(1,r)$ 
slope and then drill along the core of the surgered torus, we would 
also obtain a knot complement in $D^2\times S^1$ that admits three 
$D^2 \times S^1$ surgeries. In fact, by the above mentioned corollary, 
these are the only knots in solid tori with this property. 

The above constuction shows that Dehn surgery along a $(1,r)$ 
slope of the unknotted cusp of the
Berge manifold produces knot complements that admit three lens space 
surgeries. In fact, it is well known that the $(1,0)$, $(18,1)$ and $(19,1)$ 
surgery slopes on the $(-2,3,7)$ pretzel knot admit lens space surgeries
(see \cite{FS}). By drilling out the unknotted cusp of the Berge 
manifold, these are also the 
surgery slopes that produce a solid torus filling.
Since the linking number of the knotted cusp and the unknotted cusp is
$7$, the longitude gets sent to the curve 
$(49r,1)$ after $(1,r)$ Dehn surgery on 
the unknotted cusp while the meridian $(1,0)$ remains 
fixed (see \cite[Sect 9.H]{Rolfsen}). So the $(1,0)$, $(18,1)$, 
and $(19,1)$ surgery parameters get sent to 
$(1,0)$, $(49r+18,1)$, and $(49r+19,1)$ respectively after
$(1,r)$ Dehn surgery on the unknotted cusp. Furthermore, we can use the 
surgery paramters to compute the homology of the manifolds resulting from
lens space surgeries on the knot complements. 
In fact, we see that for these knots we obtain 
$S^3$ and two lens spaces - one with 
fundamental group of order $|49r+18|$ 
and another of order $|49r+19|$.

More generally, if we allow Dehn surgery
along any $(p,q)$ slope
of the unknotted cusp of the Berge manifold where (p,q)=1,
and either $(1,0)$, $(18,1)$, or $(19,1)$ Dehn surgery on the knotted cusp, 
we will also get lens spaces.
Again, by \cite[Sect 9.H]{Rolfsen}, we see that the $(1,0)$ 
surgery slope corresponds to a lens space of 
order $|p|$, $(18,1)$ surgery slope corresponds to a lens space of 
order $|49q+18p|$, and $(19,1)$ surgery slope
corresponds a lens space of order $|49q+19p|$.

\subsection{} 
Denote 
$v_0\approx 1.01494146$ as the volume 
of the regular ideal tetrahedron. The
Berge manifold is comprised of 
four such tetrahedra and therefore
its volume is $4v_0$. Denote by $\Gamma_L$ as 
the fundamental group of the Berge manifold.
Since the complement of the Berge manifold is comprised of four
regular ideal tetrahedra, 
$\Gamma_L \subset Isom^+(\mathbb{T})\cong \PGLTOT$,
where $\mathbb{T}$ is a tesselation of 
$\Hth$ by regular ideal tetrahedra. 
Hence, the Berge manifold is arithmetic.

The proof of the following lemma takes advantage of the fact that the Berge
manifold has relatively low volume in order to show that  
it cannot cover
an orbifold with a torus cusp and a rigid cusp. Where necessary, we consider 
all groups as subgroups in $\PSLTC$.

\begin{lem}\label{coveringLemma}
The Berge manifold does not cover an orbifold with a torus cusp and a 
rigid cusp.

\end{lem}

\begin{proof}[Proof of \ref{coveringLemma}]
Assume $Q_T$ is an orbifold with a torus cusp and a rigid cusp covered by
the Berge manifold. 
Since the invariant trace field of the Berge manifold is
$\Q(\sqrt{-3})$, the rigid cusp of $Q_T$ must be either 
$S^2(3,3,3)$ or $S^2(2,3,6)$. 
In either case, consideration of the unknotted torus cusp 
of the Berge manifold covering the rigid cusp shows the 
degree of such a cover is $3k$ for
some integer $k\geq 1$. 
Also, since 
the Berge manifold is arithmetic and
the class number of 
$\Q(\sqrt{-3})$ is 1, it follows from \cite[Thm 1.1]{CLR}, that any maximal 
group commensurable with 
the Berge manifold has exactly one cusp. Thus, there exists 
a one-cusped orbifold $Q_M$ covered by $Q_T$. 
By consideration of the cusps of $Q_T$ 
covering the rigid cusp of $Q_M$ (see Prop \ref{prop_degree}),
we see that the covering degree 
of such a map would be $3l+n^2$ or $3l+2n^2$ for some integers $l,n$ (In the 
later case, $l$ must be even).

Thus, the covering of $Q_M$ by the Berge manifold
is of order $d=3k(3l+n^2)$ or $d=3k(3l+2n^2)$.  Now,
$d\leq 48$ (see  \cite{Me}) and  since $k,l,n \geq 1$, 
we have that $d \geq 12$. Hence, $vol(Q_M) \leq v_0/3$ if $Q_M$ has a $S^2(3,3,3)$ cusp 
and
$vol(Q_M) \leq v_0/6$ if $Q_M$ has a $S^2(2,3,6)$ cusp.

\begin{figure}[ht]
\begin{center}
\includegraphics[height = 1.8 in]{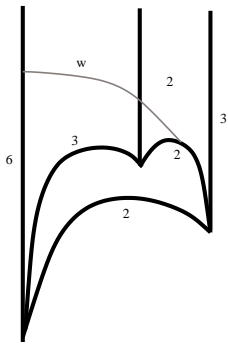}
\caption{\label{tetDomain} The fundamental domain for $\Gamma$ together
with the involution $w$}
\end{center}
\end{figure}

It follows that this orbifold must appear 
on the lists in \cite[Thm 3.3, 4.2]{A1} and \cite{NR2}.  
However, none of the orbifolds with $S^2(3,3,3)$ cusps
appearing on these lists correspond to maximal groups 
commensurable with the Berge manifold,
so we may assume that $Q_M$ has a $S^2(2,3,6)$ cusp. 
After combining the above restrictions on the degree of a 
cover and the restrictions from Adams' list, there 
are two possiblities for $Q_M$:

 \emph{either 
$Q_M$ has volume $v_0/6$ and a $S^2(2,3,6)$ cusp (here $k=1$, $l=2$, $n=1$) 
or 
$Q_M$ has volume $v_0/12$ and a $S^2(2,3,6)$ cusp
(here $k=2$, $l=2$, $n=1$)}.

First, consider the case where $Q_M$ has volume $v_0/6$. 
By noting that $\pi_1 ^{orb}(Q_M)$ has an index 2 subgroup 
$\Gamma :=<x,y,z | x^2,y^2,z^3,(yz^{-1})^2,(zx^{-1})^6,(xy^{-1})^3>$ and $\pi_1 ^{orb}(Q_M) =
<\Gamma, w>$ where $w$ is the order 2 rotation on the fundamental domain of $\Gamma$,  
we obtain a presentation for $\pi_1 ^{orb}(Q_M)$ 
 (see \cite{NR1}, \cite{MR} and Figure \ref{tetDomain}). 

Thus, we obtain the following presentation 
$$\pi_1^{orb}(Q_M)= <w,x,y,z | x^2,y^2,z^3,w^2,(yz^{-1})^2,(zx^{-1})^6,(xy^{-1})^3, (wx)^2,wywyz^{-1}>.$$
However, using GAP, the above group does not have any index 8 subgroups. 
Thus, there can be no orbifold $Q_T$.                                    

In second case, $Q_M\cong \Hth/\PGLTOT$ and 
the $[\PGLTOT:\pi_1^{orb}(Q_T)]=8$. 
If $\pi_1^{orb}(Q_T) \subset \PSLTOT$, $[\PSLTOT:\pi_1^{orb}(Q_T)]=4$. 
Using GAP, there is a unique index 4 subgroup $G$
of $\PSLTOT$. However, $G$ has finite abelianization, and therefore
cannot be the orbifold group of $Q_T$.

Thus, we may assume that 
$\pi_1^{orb}(Q_T) \not\subset \PSLTOT$ and deduce that there is 
a unique subgroup $\Lambda$ of index 2 in $\pi_1^{orb}(Q_T)$ 
such that $\Lambda \subset \PSLTOT$. By covolume considerations
 $\Lambda$ has index 8 in 
$\PSLTOT$. Also, $\Hth/\Lambda$ has a torus cusp and an $S^2(3,3,3)$ cusp.  
Since $\Hth/\PSLTOT$ 
has an $S^2(3,3,3)$ cusp, the degree of the covering $p:\Hth/\Lambda\rightarrow \Hth/\PSLTOT$ has to be $3l+n^2$ (see Prop \ref{prop_degree}), which is never 8.

This completes the proof.
\end{proof}

\section{Cyclic Surgeries on $\bnm$}\label{threeCyclicSurgeries}

In this section, we show that for fixed n and m, $\bnm$ admits three 
finite cyclic surgeries. We also show directly it is 
covered by three knot complements if $n \ne 7$. 

\begin{lem}\label{setupbetan}
The orbifolds $\bnm$ are covered by three knot complements. Further more, the
degrees of the corresponding covering maps are distinct.
\end{lem}

\begin{proof}
For a fixed $\bnm$, let $r=(n,m)$ and consider $\bnm$ as the 
union of the complement of a knot in a solid torus, $T_1$  and 
a solid torus with core a singular locus of order r, $T_2$ 
(see Figure \ref{twoTori}).

\begin{figure}[htp]
\begin{center}
\includegraphics[height = 1.5 in]{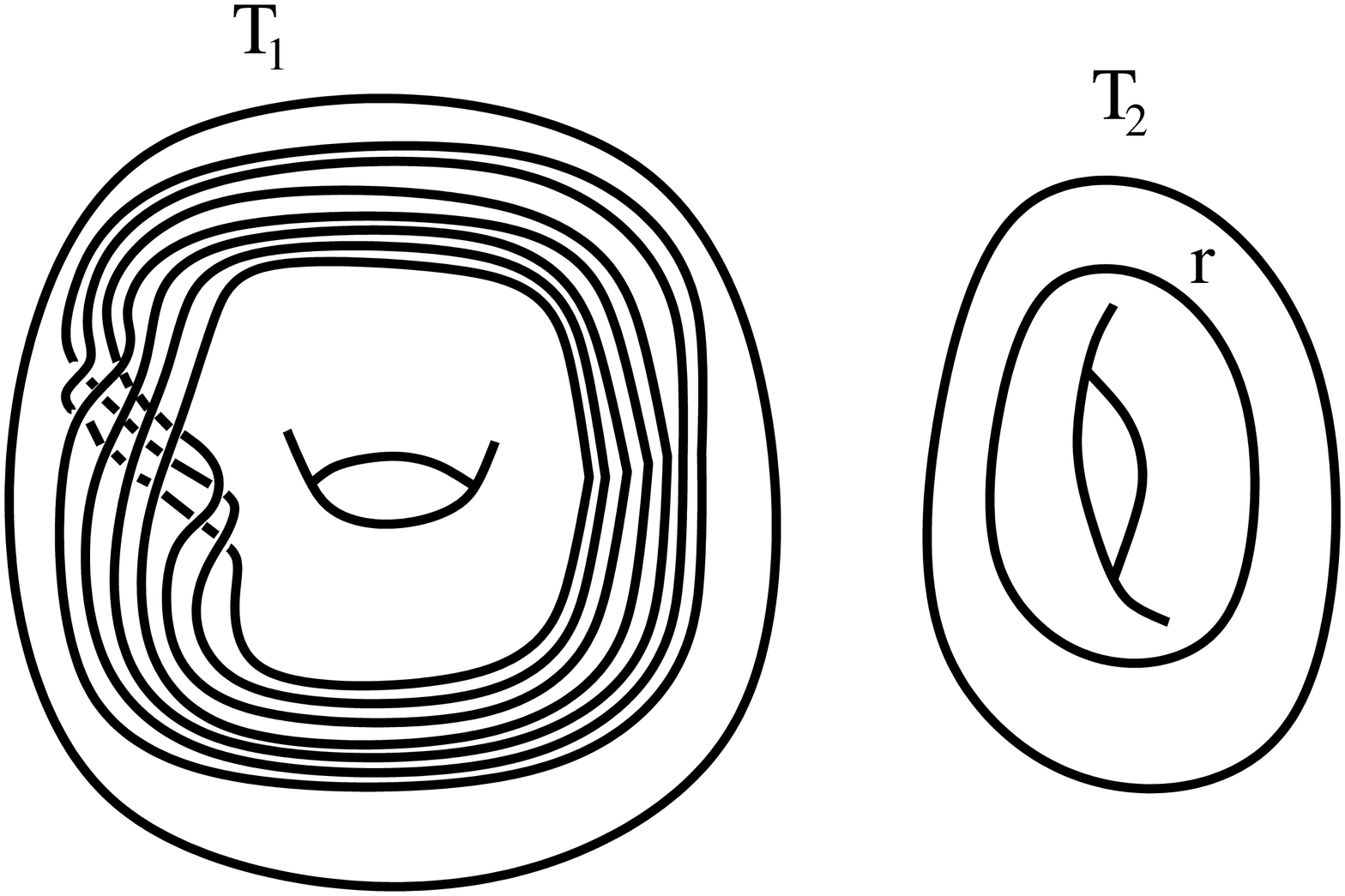}
\caption{\label{twoTori} The decomposition of a surgered $\bnm$ along a torus }
\end{center}
\end{figure}

By \cite[Cor 2.9]{Berge}, $T_1$  admits three Dehn surgeries 
that result in a solid torus. 
Thus, $\bnm$ admits three Dehn surgeries that are homeomorphic 
to $T_2$ and a solid torus
glued together along their boundaries. Each orbifold 
$O_j$ $(j \in \{1,2,3\})$ resulting
from one of these Dehn surgeries has 
underlying space a lens space with $\pi_1 ^{orb} (O_j)$ 
finite cyclic. 

In fact, $|\pi_1^{orb}(O_j)|$ is distinct for each choice of $j$. 
To see this we observe, as noted above, that $O_j$ 
is an orbifold with underlying space a lens space.  
Moreover, this underlying space is a lens space with 
fundamental group of order either
$\frac{n}{r}$, $|49\frac{m}{r}+18\frac{n}{r}|$, or $|49\frac{m}{r}+19\frac{n}{r}|$ depending on the choice of surgery on $T_1$ 
(see $\S$ \ref{prelim}). 
Splitting $O_j$ into a solid torus coming 
from the Dehn surgery on $T_1$ and $T_2$ the solid torus core a singular curve,
we can compute $\pi_1 ^{orb}(O_j)$ using van Kampen's theorem. 
Thus, the orders of the each fundamental group increase by a factor of r 
and $|\pi_1 ^{orb}(O_j)|$ is either $n$, $r\cdot |49\frac{m}{r} + 18\frac{n}{r}|$ or 
$r\cdot |49\frac{m}{r} + 19\frac{n}{r}|$ 
which take on three distinct values for fixed $n$, $m$ and $r$.

In addition, by the Orbifold Theorem (see \cite[Thm 2]{BP}) and the above argument that $\pi_1 ^{orb}(O_j)$ is finite cyclic, each $O_j$ has $S^3$ as its universal cover.  Denote this 
covering map $\phi_j : S^3 \rightarrow O_j$. 
We may view $O_j$ as the union of the solid torus 
torus coming from the cusp Dehn filling of $\bnm$ and the complement
of this solid torus, which we denote by $B$. 

Hence $\phi_j^{-1}(B)$ is a knot
or link exterior in $\Sth$. Since $(n,7)=1$ and the 
singular set of $T_2$ has linking number 7 with the 
knotted cusp of $\bnm$, the boundary of
$\phi_j ^{-1}(B)$ is connected. 
Hence, if $(n,7)=1$, $\bnm$ will be covered by three knot complements 
in $\Sth$. Also, since the orders of $|\pi_1^{orb}(O_j)|$ are distinct,
the covering degree of $\phi_j$ will take on a distinct value for each $j$. 

\end{proof}

\begin{rem}
When $n=1$, the classification of exceptional Dehn surgeries in 
\cite[Table A.1, Rem A.3]{MP} shows that $\bnm$
is hyperbolic. Hence, $\beta_{1,m}$ is a hyperbolic knot complement 
that admits three cyclic surgeries.
\end{rem}

\section{Proof of The Main Theorem}\label{mainproof}

In this section, we prove Theorem \ref{mainthm1}. 
Also for this section, we consider $\onm$, $\Delta_{n,m}$, 
and $\Omega_L$ as subgroups of $\PSLTC$.

\begin{proof}[Proof of Theorem \ref{mainthm1}]

Using Lemma \ref{setupbetan}, each $\bnm$ is covered by three 
knot complements such that the covers are of distinct degrees. 
Also, the Hyperbolic Dehn Surgery Theorem 
\cite[Thm 5.8.2]{Thurston1} 
shows that all but at most finitely many of the $\bnm$ are hyperbolic. 
For the rest of the proof we only consider those 
$\bnm$ that are hyperbolic. Given this condition, each $\bnm$
we consider is covered by 
three distinct knot complements. By \cite[Thm 1.3]{BBW}, 
to prove Theorem \ref{mainthm1} it suffices to 
show that the knot complements covering $\bnm$ 
do not have hidden symmetries.

Suppose an infinite number of the hyperbolic
knot complements that cover $\bnm$ admit 
hidden symmetries. By the discussion in $\S$\ref{prelim_cusps}, 
every such a knot 
complement will non-normally cover an orbifold $Q_{n,m}$
with a rigid cusp. Furthermore, on passage to a subset of the $\bnm$, 
we can assume that the orbifolds $Q_{n,m}$ have the same type 
of rigid cusp, $C$. Let $\onm = \pi_1 ^{orb}(\bnm)$, $\Delta_{n,m}=\pi_1^{orb}(Q_{n,m})$ and let 
$P\subset \PSLTC$ be the peripheral subgroup of 
$\Delta_{n,m}$. We may assume that each $\onm$ is 
conjugated so that $P$ has a fixed representation in $\PSLTC$. 
Since $\bnm$ has one cusp, notice that $\Delta_{n,m} =P\cdot\onm$

By Thurston's Hyperbolic Dehn Surgery Theorem 
\cite[Thm 5.8.2]{Thurston1}, the volumes of the $\bnm$ 
are bounded from above by the volume of the Berge
manifold. In addition, the minimum 
volume of a non-compact oriented hyperbolic
3-orbifold is $\frac{v_0}{12}$ (see \cite{Me}). 
Hence, $vol(Q_n)\geq\frac{v_0}{12}$.  Thus, we can further 
subsequence to arrange that $\bnm$ covers $Q_{n,m}$,
that the $Q_{n,m}$'s have the same type of rigid cusp, 
and that the covering degree is fixed, say d.

Since $\bnm $ is obtained by Dehn surgery on 
the Berge manifold, the $\onm$ will converge 
algebraically and geometrically to $\Omega_L$, the fundamental
group of the Berge manifold
(see \cite[Thm 5.8.2]{Thurston1}).
As P was a fixed group in our construction,  
$\Delta_{n,r}$ also converges algebraically and geometrically to $P\cdot\Omega_L.$

We have the following diagram:

$$\xymatrix{
\Delta_{n,m}  \ar[r]^{(n,m)\rightarrow \infty} & P\cdot\Omega_L\\
\onm \ar@{^{(}->}[u]_d \ar[r]^{(n,m)\rightarrow \infty} & \Omega_L \ar@{^{(}->}[u]_d
}$$

Note, $[P\cdot\Omega_L: \Omega_L] =d< \infty$. 
Let $Q_T=\Hth/P\cdot\Omega_L$. $Q_T$ has two cusps: a torus cusp, 
corresponding to the cusp created by 
geometric convergence from Dehn surgery,
and a rigid cusp, corresponding to the
cusp with peripheral group $P$.

However by Lemma \ref{coveringLemma}, such a limiting $Q_L$ cannot exist.
Hence, at most finitely many of the $\bnm$ have hidden symmetries.

\end{proof}

\begin{rem}
To find explicit examples of hyperbolic knot complements with three knot 
complements in the commensurability class, 
we can use the computer program snap to show directly that there are no hidden 
symmetries. Specifically,
for m=0 and n=2,3,4,5,6,7, 
$\bnm$ is hyperbolic and snap show us that $\bnm$ has 
an invariant trace field with real embeddings.
These fields cannot contain $\Q(i)$ or $\Q(\sqrt{-3})$ 
as subfields. Thus, these knot complements do not have hidden 
symmetries (recall $\S$ \ref{prelim_cusps}) and there are exactly three knot complements 
in the commensurability classes.

\end{rem}

\section{Remarks}\label{remarks}

The following theorem provides a partial classification of hyperbolic orbifolds 
covered by three knot complements. It can be seen as a direct corollary to a 
result of \cite{BBW}. However, a proof is provided below for completeness.

\begin{thm}\label{theseAreTheOnlyOnes}
Let $O$ be a closed 3-orbifold and let $K$ be a knot in $O$ that 
is disjoint from singular locus of $O$. If $O-K$ is:
\begin{enumerate}
\item hyperbolic,
\item covered by 3 knot complements, 
\item does not admit hidden symmetries, and
\item $O$ has non-empty singular locus,
\end{enumerate}

then $O-K \cong \bnm$ for some pair $(n,m)$.
\end{thm}

\begin{proof}
Let $\gamma$ be the singular locus of $O$. Denote $|O|$ the underlying
space of $O$. By \cite[Thm 1.2]{BBW} 
and the assumptions, we know that $|O|$ is a lens space, $\gamma$ is  
a non-empty subset of the cores of a genus 1 Heegaard splitting of $|O|$, 
and if $S^3-K$ covers $O-K$ then it does so cyclically and corresponds 
to a finite cyclic filling of $O-K$. Finally, denote
$M=O-\gamma-K$

First assume $\gamma$ has one component.
Each of the three knot complements covering $O-K$ will correspond 
to a $S^1\times D^2$ filling on knotted cusp of $M$. 
Again, we appeal to the fact that there is a a unique family
of knots in solid tori that admits 3 non-trivial $S^1 \times D^2$
fillings (see  \cite[Cor 9.1]{Berge}). Hence, M is obtained by performing
$(1,m)$ surgery on the unknotted cusp of the Berge manifold then drilling out
the core of the surgered torus. Gluing back in the neighborhood 
of the fixed point set of $\langle \gamma \rangle$ gives us $\bnm$ for
some $n,m$.

Now, assume that $\gamma$ has two components $\gamma_1$ and $\gamma_2$.
$M=T^2\times I - K'$, where $K'$ is a knot. 
Each of the three finite cyclic 
on $O-K$ corresponds $M$ admitting a $T^2\times I$ filling. Hence,
Dehn filling along the cusp corresponding to $\gamma_1$ will produce a knot 
complement in $D^2\times S^1$ with three $D^2\times S^1$ fillings.

Denote $l_1$ to be the linking number of
$\gamma_1$ and $ K'$ and $l_2$ to be the linking number of 
$\gamma_2$ and $K'$. If $l_1$ is zero, 
$K'$ would be a knot in a solid torus that is not a 1-braid after
$(1,0)$ on $\gamma_2$ but has two non-trivial $S^1\times D^2$ fillings. 
This contradicts
 \cite[Cor 9.1]{Berge}. Hence, we may assume $l_1\ne 0$ and 
$l_2\ne 0$. 

Also, $(1,n)$ surgery on $\gamma_2$ will produce a knot $K''$ in a solid torus
that has linking number $l_2 + n\cdot l_1$ with $\gamma_2$. In particular for 
large enough $n$  $l_2 + n\cdot l_1 \ne 7$. Hence, in cannot be in the 
family of knots that admit two non-trivial $S^1\times D^2$ fillings.  
\end{proof}

One might hope to relax condition $(4)$ above. However, Brandy Guntel 
pointed out that the $K(7,5,2,-1)$ knot complement (see Figure
\ref{BrandysKnot}) is hyperbolic 
and admits two non-trivial cyclic surgeries. The fundamental group 
of one of these lens spaces is of order 32. By our original discussion in $\S$\ref{bergeDef}, knot complements obtained by Dehn surgery on the unknotted cusp of the Berge manifold have lens spaces of order $|49r-18|$ and $|49r-19|$ neither of which can be 32. Hence, the $K(7,5,2,-1)$ complement is not one of the $\bnm$. 
However, since the invariant trace field of the $K(7,5,2,-1)$ is an odd degree
extension of $\Q$, we see that this knot complement does not admit hidden symmetries and the $K(7,5,2,-1)$ has exactly three knot complements in its comensurability class (see \cite[Cor 5.4]{RW}). 

\begin{figure}
\begin{center}
\includegraphics[height = 1.5 in]{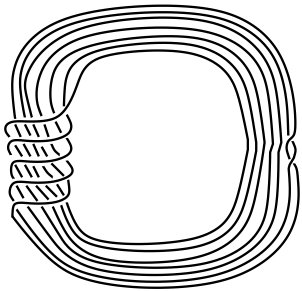}
\caption{\label{BrandysKnot} The K(7,5,2,-1) }
\end{center}
\end{figure}

As mentioned above $(1,m)$ surgery on 
the unknotted cusp of the Berge manifold
produces Berge knots. It seems natural to ask if any 
hyperbolic Berge knots can have hidden 
symmetries. More generally, we might ask if any 
hyperbolic knot complements can 
have hidden symmetries and admit non-trivial lens space surgeries. 
As discussed in $\S$
\ref{intro}, there are three hyperbolic 
knot complements known to have hidden symmetries: the complements of the 
two dodecahedral knots
of Aitchison and Rubinstein, and the figure eight knot complement
(see \cite{AR},\cite{NR1}). 
Using SnapPea one can see that both dodecahedral knots are amphichiral.
Thus, by \cite[Cor 4]{CGLS} they cannot admit a lens space surgery.  
Additionally, it is well known that the figure eight knot complement
does not admit a lens space surgery (see 
\cite{Takahashi} for example).

\section{Acknowledgments}
First, I would like to thank Alan Reid for raising the questions that 
lead to this paper and thoughtfully guiding this work from its 
formative stages to completion.
I also would like to thank Cameron Gordon for showing
me the family of knot complements with three lens space surgeries
$\beta_{1,m}$. Third, I would like to thank Genevieve Walsh and Steve Boyer
for a number of enlightening conversations and their 
suggestions on early versions of this paper.   
Finally, I would like to thank my 
fellow graduate students for a number of helpful conversations.


\begin{thebibliography}{99}
\bibitem[A]{A1} C. Adams, \emph{Non-compact 3-orbifolds of small volume} in 
Topology '90, Ohio State Univ. Math. Res. Inst. Publ. Vol. 1, de 
Gruyter (1992) pp. 273-310. 

\bibitem[AR]{AR} I. R. Aitchison and J. H. Rubinstein, \emph{Combinatorial cubings,
cusp and the dodecahedral knots}, in Topology '90, Ohio State Univ. Math. Res. Inst. Publ. 1, de Gruyter (1992),  pp. 17-26.

\bibitem[Be]{Berge} J. Berge, \emph{The knots in $D^2 \times S^1$ with nontrivial Dehn surgeries yielding $D^2 \times S^1$}, Topology. Appl. Vol. 38, (1991) pp. 1-19.

\bibitem[BBW]{BBW} M. Boileau, S. Boyer, and G. Walsh, \emph{On commensurability of knot complements}, preprint.

\bibitem[BP]{BP} M. Boileau, J. Porti, \emph{Geometrization of 3-orbifolds of cyclic type}, Ast\'ersique 272 (2001).

\bibitem[CD]{CD} D. Calegari and N. M. Dunfield, \emph{Commensurability of 1-cusped hyperbolic 3-manifolds}, Trans. Amer.
Math. Soc. Vol. 354 (2002), pp. 2955-2969. 

\bibitem[CGLS]{CGLS} M. Culler, C.McA Gordon, J. Leucke and P. Shalen, \emph{Dehn surgery on knots}, Ann. of Math. Vol. 125 (1989) pp. 237-300.

\bibitem[CLR]{CLR} T. Chinburg, D. Long, and A. W. Reid, \emph{Cusps of 
minimal non-compact arithmetic hyperbolic 3-orbifolds}, Pure and Applied Math Quarterly, Vol. 4 (2008) pp. 1013-1031.

\bibitem[FS]{FS} R. Fintushel and R. Stern, \emph{Constructing lens spaces by
surgery on knots}, Math. Z. Vol. 175 (1980) pp. 33-51. 

\bibitem[GHH]{GHH} O. Goodman, D. Heard, and C. Hodgson, \emph{Commensurators of Cusped Hyperbolic Manifolds}, Experimental Mathematics, Vol. 17 no 3 (2008), pp 283-306.

\bibitem[GW]{GW} F. Gonz\'{a}lez-Acu\~{n}a and W.C. Whitten, \emph{Imbeddings of three-manifold groups}, Mem. Amer. Math. Soc. Vol. 474 (1992).

\bibitem[HS]{HS} J. Hoste and P. Shanahan, \emph{Commensurability classes of twist knots},  J. Knot Theory Ramifications  Vol. 14 no. 1,(2005) pp. 91-100.

\bibitem[MR]{MR} C. Maclachlan and A. W. Reid, \emph{The Arithmetic of Hyperbolic 3-Manifolds}, Springer New York (2003).

\bibitem[MP]{MP} B. Martelli and C. Petronio, \emph{Dehn filling of the ''magic'' 3-manifold},  Comm. Anal. Geom. Vol. 14, No. 5 (2006), p. 969-1026.

\bibitem[Me]{Me} R. Meyerhoff, \emph{The cusped hyperbolic 3-orbifold of minimum volume}, Bull. Amer. Math. Soc. Vol. 13 (1985) pp. 154-156. 

\bibitem[MB]{MB} J. W. Morgan and H. Bass (editors), 
\emph{The Smith Conjecture}, Academic Press, New York (1984). 


\bibitem[MM]{MM} M. Macasieb and T. Mattman, \emph{Commensurability classes of (-2,3,n) pretzel knot complements}, Alg. \& Geom. Top.  8(3) (2008), pp. 1833-1853. 

\bibitem[NR1]{NR1} W. D. Neuman and A. W. Reid, \emph{Arithmetic of hyperbolic manifolds}, in Topology '90, Ohio State Univ. Math. Res. Inst. Publ. 1, de Gruyter (1992),  pp. 273-310.

\bibitem[NR2]{NR2} W. D. Neuman and A. W. Reid, \emph{Notes on Adams' small volume orbifolds}, in Topology '90, Ohio State Univ. Math. Res. Inst. Publ. 1, de Gruyter (1992), pp. 311-314.

\bibitem[Re]{Reid} A. W. Reid, \emph{Artimeticity of knot complements}, J. London Math. Soc. (2) (1991), pp. 171-184.

\bibitem[RW]{RW} A. W. Reid, G. S. Walsh, \emph{Commensurability classes of 2-bridge knot complements}, Alg. \& Geom. Top.  Vol. 8 No. 2 (2008), pp. 1031-1057.

\bibitem[Ro]{Rolfsen} D. Rolfsen, \emph{Knots and Links.} Publish or Perish Press, Berkeley, (1976). 

\bibitem[Ta]{Takahashi} M-o Takahashi, \emph {Two-bridge knots have property P}, Mem. Amer. Math. Soc. 29 (1981)

\bibitem[Th]{Thurston1} W. Thurston, \emph{The geometry and topology of 
3-manifolds}, Princeton University, 1977, Mimeographed lecture notes.

\bibitem[WZ]{WZ} S. Wang and Q. Zhou, \emph{Symmetry of knots and cyclic surgery}, Trans. Amer. Math. Soc. Vol. 330 No. 2 (1992), 
pp. 665-676.

\end{thebibliography}
\end{document}